\documentclass{amsproc}

\usepackage{amssymb}

\usepackage{graphicx}

\usepackage[cmtip,all]{xy}

\usepackage{amscd,amsmath,amsfonts}
\usepackage{xcolor,color}
\usepackage{hyperref}
\usepackage{url}
\usepackage{mathtools}

\newtheorem{theorem}{Theorem}[section]
\newtheorem{lemma}[theorem]{Lemma}
\newtheorem{prop}[theorem]{Proposition}
\newtheorem{cor}[theorem]{Corollary}

\theoremstyle{definition}
\newtheorem{definition}[theorem]{Definition}

\theoremstyle{remark}
\newtheorem{remark}[theorem]{Remark}
\newtheorem{nota}[theorem]{Notation}

\numberwithin{equation}{section}

\newcommand{\ml}[2]{\begin{multline}\label{#1}#2 \end{multline}}
\newcommand{\ga}[2]{\begin{gather}\label{#1}#2 \end{gather}}

\newcommand{\Spec}{{\rm Spec \,}}

\newcommand{\sF}{{\mathcal F}}
\newcommand{\sG}{{\mathcal G}}
\newcommand{\sH}{{\mathcal H}}
\newcommand{\sI}{{\mathcal I}}

\newcommand{\sK}{{\mathcal K}}
\newcommand{\sL}{{\mathcal L}}

\newcommand{\sO}{{\mathcal O}}

\newcommand{\sS}{{\mathcal S}}

\newcommand{\A}{{\mathbb A}}

\newcommand{\F}{{\mathbb F}}
\newcommand{\G}{{\mathbb G}}

\newcommand{\N}{{\mathbb N}}
\renewcommand{\P}{{\mathbb P}}
\newcommand{\Q}{{\mathbb Q}}

\newcommand{\Z}{{\mathbb Z}}

\begin{document}

\title{ Bounding ramification by covers and curves}


\author{H\'el\`ene Esnault}
\address{Freie Universit\"at Berlin, Arnimallee 3, 14195, Berlin,  Germany}
\curraddr{}
\email{esnault@math.fu-berlin.de}
\thanks{Supported  during part of the preparation of the article by the Institute for Advanced Study, USA}

\author{ Vasudevan Srinivas}
\address{TIFR, School of Mathematics\\Homi Bhabha Road\\400005 Mumbai, India }
\curraddr{}
\email{srinivas@math.tifr.res.in}
\thanks{Supported by a J. C. Bose Fellowship of the Department of Science and Technology, India and the Department of Atomic Energy, Government of India, under project number 12-R\&D-TFR-RTI4001}


\date{}

\begin{abstract}
We prove that $\bar \Q_\ell$-local systems of bounded rank and ramification on a smooth variety  $X$ defined over an algebraically closed field $k$ of characteristic $p\neq \ell$ are tamified outside of codimension $2$ by a finite separable cover of bounded degree.   In rank one, there is a curve which preserves their monodromy.  There is a curve defined over the algebraic closure of a purely transcendental extension of $k$ of finite degree which fulfills the Lefschetz theorem. 

\end{abstract}

\maketitle


\section{Introduction} \label{intro}

 The notions of fundamental group introduced by Poincar\'e and Riemann for topological manifolds and of Galois group of field extensions introduced by Galois were unified by Grothendieck's theory of \'etale fundamental groups.  For  complex varieties, by the Riemann existence theorem, the \'etale fundamental group is the profinite completion of the topological one, so  both groups share many properties. It is no longer the case for a variety defined in characteristic $p>0$, due to wild ramification.  In this note we show how to give various 'upper bounds' of ramification.

Let $X$ be a smooth connected variety  of finite type over an algebraically closed field  $k$ of characteristic  $p>0$.  Let $\sF$ be a $\bar \Q_\ell$ local system. It is defined by a continuous representation $\rho: \pi_1^{\rm \acute{e}t}(X,x)\to GL_r(\bar \Q_\ell)$ from the \'etale fundamental group  $\pi_1^{\rm \acute{e}t}(X,x)$ based in one geometric point $x$. Choosing  a lattice stabilized by $\rho$ defines a residual representation $\bar \rho: 
\pi_1^{\rm \acute{e}t}(X,x)\to GL_r(\bar \F_\ell)$ which by continuity has values in $GL_r(\F)$ for a finite extension $\F\supset \F_\ell$. It is well defined modulo semi-simplification.  The Galois \'etale cover of $X_{\bar \rho} \to X$ defined by $\bar \rho$ has the property that the pullback
$\sF|_{X_{\bar \rho}}$ 
 of $\sF$ to $X_{\bar \rho}$ is tame. We say that it {\it tamifies} $\sF$.  If we bound the ramification of $\sF$ by an effective Cartier divisor $D$ supported on $\bar X\setminus X$ where $j: X\hookrightarrow \bar X$ is a normal compactification in the sense  of \cite[Defn. 4.6]{EK11},  \cite[2.1]{EK12}, and we bound $r$,  the degree of $X_{\bar \rho} \to X$ is not bounded.  Indeed $D=0$ is equivalent 
to  $\rho$ factoring over the tame quotient $\pi_1^{\rm \acute{e}t,t}(X,x)$ of $\pi_1^{\rm \acute{e}t}(X,x)$.
  Already in this case, the degree   is not bounded.    For example, the pro-$\ell$-completion of the tame fundamental group of $\P^1\setminus \{0,1,\infty\}$ is by Grothendieck's specialization theorem the pro-$\ell$-completion of the free group on two letters, which is represented in $GL_2(\F_{\ell^n})$ for unbounded $n$. 
 
 \medskip
 
We assume that $X$ is quasi-projective. 
  Let $\pi: Y\to X$ be a normal connected finite cover.
  We say that  $\pi$     {\it tamifies $\sF$ outside of codimension $2$} if there is a normal compactification $Y\hookrightarrow \bar Y$
  and a 
   closed subset $ \Sigma\subset \bar Y$ of codimension $\ge 2$ such that $\sF|_{  Y }$   is tame along $(\bar Y \setminus Y) \setminus \Sigma$.  
We prove the following theorem.

\begin{theorem}[Boundedness theorem] \label{thm:bdd} Given $j, X, r, D$, there is a natural number $M$ such that for any $\sF$ of rank and ramification bounded by $(r,D)$,  there is a finite generically \'etale cover $X_{\sF}\to X$ of $X$ of degree $\le M$ which tamifies $\sF$ outside of codimension $2$. 
\end{theorem}

If $k$ is a finite field,  by L. Lafforgue's theorem in dimension $1$ and Deligne's theorem \cite[Thm.1.1]{EK12} in general, up to twist by a character of $k$ there are finitely many isomorphism classes of  simple $\bar \Q_\ell$ local systems $\sF$ in rank and ramification bounded by $(r,D)$. This is analogous to Hermite-Minkowski theorem in number theory according to which there finitely many isomorphism classes of extensions of a given number field with bounded degree and bounded discriminant.
 As a consequence,   there is a smooth curve $C\hookrightarrow X$ such that 
$\sF|_C$ 
 remains irreducible in bounded $(r,D)$  (\cite[Prop.B.1]{EK12}, see other references therein).  Deligne asked in \cite{Del16} whether over an algebraically closed field,   there is a smooth curve $C\hookrightarrow X$ such that for any $\sF$  with  bounded  $(r,D)$, $\sF|_C$  keeps the same monodromy group.   To understand the question, recall that by Drinfeld's theorem \cite[Prop.~C2]{Dri12} there is a full Lefschetz theorem for the tame fundamental group, that is there is such a curve for which the functoriality homomorphism $\pi_1^{\rm \acute{e}t,t}(C,x)\to
  \pi_1^{\rm \acute{e}t,t}(X,x)$ is surjective. On the other hand, as is well known, there is no Lefschetz theorem for $ \pi_1^{\rm \acute{e}t}(X,x)$ (see e.g. \cite[Lem~5.4]{Esn17}). 
 So Deligne's problem asks whether we can save the Lefschetz theorem bounding $D$, but also in addition $r$ (this condition being unnecessary for $D=0$) without any arithmeticity assumption.       For $r=1$ this is true if $j$ is a good normal crossings compactification by Kerz-Saito's theorem,  
see  \cite[Thm.~1.1]{KS14} and erratum to come. We give a complete answer to Deligne's question in rank one. 
\begin{theorem}[Good curve in rank one] \label{thm:rankone} There is a smooth $C\hookrightarrow X$ such that for any $\sF$  of rank one with ramification  bounded by $D$, $\sF|_C$  keeps the same monodromy group.  
\end{theorem}
In fact, a slightly more precise statement is true, see Remark~\ref{rmk:deligne}. 
 Theorem~\ref{thm:bdd} has the following corollary. 
\begin{cor} \label{cor:bdd}
Given  $j, X, r, D$, there is a natural number $M$ such that
for any $\sF$ of rank and ramification bounded by $(r,D)$, 
 there is a smooth curve $C\hookrightarrow X$  and a finite generically \'etale cover $C'\to C$ of degree $\le M$ such that $\sF|_C$  has the same monodromy group as $\sF$, $C'$ is integral,   and  $\sF|_{C'}$ is tame and has the same monodromy group as $\sF|_{X_{\sF}}$.

\end{cor}

Theorem~\ref{thm:bdd}, Corollary~\ref{cor:bdd}   and Theorem~\ref{thm:rankone}   give  some evidence for a general positive answer to Deligne's problem.

\medskip

The proof of Theorem~\ref{thm:bdd} consists in globalizing the arguments used to prove \cite[Prop.2]{EKS19}. To this aim, by a standard argument we reduce the problem to $X\hookrightarrow \bar X$ being $\A^d\hookrightarrow \P^d$.  There, one tool used is Harbater-Katz-Gabber local to global extension \cite[Thm.1.4.1]{Kat86}, \cite{Har80}. The proof of Corollary~\ref{cor:bdd} relies on Drinfeld's tame Lefschetz theorem {\color{red}  {\it loc. cit}}.
The proof of Theorem~\ref{thm:rankone} relies on the same statement if $\bar X$ is smooth, together with some finiteness theorem on Frobenius invariant submodules of some local cohomology groups resting on \cite[Rmk.~4.4]{Smi94}.
\medskip

Finally we prove that a Lefschetz theorem without boundedness exists if we allow the wished curve $C$ to be defined over the algebraic closure of a purely transcendental extension of $k$ of finite degree, see Remark~\ref{rmk:lefschetz}. The proof relies purely on  the classical Bertini theorem and ought to be well known. We wrote it as we could not find a reference.

\medskip
\noindent
{\it Acknowledgement:}  The first author thanks Pierre Deligne for sending the email \cite{Del16}. We thank Haoyu Hu, Takeshi Saito and Enlin Yang
for promptly answering our questions on the various notions of Swan conductor, and Moritz Kerz for a discussion on \cite{KS14}.

\section{Elementary properties of $\sS(j,r,D)$}
Let  $k$ be an algebraically closed field of characteristic  $p>0$, $X$ be a smooth connected variety  of finite type over $k$, $j: X\hookrightarrow \bar X$   be a normal compactification, $D$ be an effective Cartier divisor supported on $\bar X\setminus X$,  $r$ be a positive natural number. Recall (see \cite[Defn. 4.6]{EK11}, used in \cite[2.1]{EK12}) that an $\ell$-adic local system $\sF$ has  {\it ramification bounded by} $D$ if for any morphism of a smooth  connected projective curve $\bar \iota: \bar C\to \bar X$  with $\iota^{-1}(X)\neq\emptyset$, the pullback to $\sF$ to $\bar \iota^{-1}(X)=:C$ has Swan conductor bounded by $\bar \iota^{-1}D$  on $\bar C$. 
\begin{nota} \label{nota:S}
We denote by $\sS(j,r,D)$ the set of isomorphism classes  of  $\bar \Q_\ell$  local systems  of rank $\le r$ with ramification bounded by $D$.
\end{nota} 

 We denote by $\pi_1^{{\rm{ \acute{e}t}}}(X,x)$ the \'etale fundamental group based at a geometric point $x$. If there is no confusion, we simply write $\pi_1^{\rm \acute{e}t}(X)$.

\subsection{Restriction to an open}
Let $U\hookrightarrow X$ be a dense open subscheme, $j': U \to \bar X$ the composition with $j$. 
\begin{lemma} \label{lem:rest}
Restriction to $U$ induces an injective map 
\ga{}{\sS(j,r,D) \hookrightarrow \sS(j', r, D).\notag}

\end{lemma}
\begin{proof}
This is an immediate consequence of $\pi_1^{\rm \acute{e}t}(U)\to \pi_1^{\rm \acute{e}t}(X)$ being surjective. 
\end{proof}
\subsection{Changing the compactification}
Let $j_i: X\hookrightarrow \bar X_i$ be two normal compactifications, and $D_1$ be an effective Cartier divisor supported on 
$\bar X_1\setminus X$. 
\begin{lemma}  \label{lem:cpt}

There is an effective divisor $D_2$ supported on $\bar X_2\setminus X$ such that $ \sS(j_1, r, D_1) \subset \sS(j_2, r, D_2)$.
\end{lemma}
\begin{proof}

Let $j_3: X \hookrightarrow \bar X_3$ be a normal compactification which dominates $j_i, \ i=1,2$. Then for any curve $C\to X$, $\bar C\to \bar X_i$ factors through $\bar C\to \bar X_3\to X_i$. We conclude that 
\ga{}{ \sS(j_1, r, D_1)= \sS(j_3, r, D_1\times_{\bar X_1} \bar X_3) \notag}
thus for any $D_2$ such that $D_2\times_{\bar X_2}\bar X_3\supset  D_1\times_{\bar X_1} \bar X_3$ one has
\ga{}{   \sS(j_1, r, D_1) \subset  \sS(j_2, r, D_2).\notag}

\end{proof}
\subsection{Projection} Let $\bar \pi: \bar X\to \bar Y$ be a  finite morphism of degree $\delta$ such that the restriction $\pi: X=\bar \pi^{-1}Y\to Y$ is \'etale, where $Y\hookrightarrow \bar Y$ is dense open. 
Let $j_Y: Y\hookrightarrow \bar Y$ be the open embedding. 
\begin{lemma} \label{lem:proj}
\begin{itemize}
\item[1)]
There is an effective Cartier divisor $D'$ supported on $ \bar Y\setminus Y$ such that pushdown to $Y$ induces a  map 
\ga{}{ \sS(j,r,D)   \rightarrow \sS(j_Y, r \delta, D') .    \notag}
\item[2)] Given an effective divisor $D_Y$ supported on $\bar Y\setminus Y$, pullback induces a map
$\sS(j_Y, r, D_Y)\to \sS(j, r, \bar \pi^*D_Y).$
\end{itemize}
\end{lemma}
\begin{proof} We prove 1).  Let $\iota_Y: {\bar C}_Y\to \bar Y$ be a morphism of a smooth projective curve. Write $\pi^{-1}C_Y=C=\sqcup_i C_i$ for the union of irreducible components, which are disjoint as $\pi$ is \'etale. This defines a commutative square
\ga{}{\xymatrix{\ar[d]_{\pi|_{C_i}} C_i \ar[r]^{\iota_i}  & X \ar[d]^{\pi}\\
C_Y \ar[r]^{\iota_Y} & Y }\notag}
where $\pi|_{C_i}$ has degree $\delta_i$ with $\sum_i \delta_i=\delta$. 
One has
\ga{}{ \iota_Y^* \pi_*\sF =\oplus_i  (\pi|_{C_i})_* \iota_i^* \sF \notag}
By \cite[Lem.3.2,Prop.3.9]{EK12},  there is an effective Cartier divisor $\Delta$ supported on $\bar Y\setminus Y$ such that  $\pi_*\F_\ell$  has ramification bounded by $\Delta$  and 
$(\pi|_{C_i})_*\F_\ell$  by $ \iota_Y^*\Delta$.  So by the Grothendieck-Ogg-Shafarevich formula applied to  $ \iota_i^*\sF$ and 
$(\pi|_{C_i})_* \iota_i^*\sF$, the Swan conductor of  $(\pi|_{C_i})_* \iota_i^*\sF$ is bounded by $r\delta_i \Delta+ (\pi|_{C_i})_*D$. It is then enough to set 
$D'=\delta( r \Delta+ \pi_*D)$ and prove 1). As for 2), this is immediate, just writing $ \bar C\xrightarrow{\bar \iota}  \bar X \xrightarrow{\bar \pi} \bar Y$. This finishes the proof of 2).  

\end{proof}

\section{ Reduction of Theorem~\ref{thm:bdd} to the case $X=\A^d$ }
The aim of this section is to prove
\begin{prop} \label{prop:A}
If the theorem is true for $\A^d \hookrightarrow \P^d$ and any $r, D$, it is true for any $X\hookrightarrow \bar X$ with $X$ of dimension $d$  and any $r, D$.

\end{prop}
\begin{proof}
By Lemma~\ref{lem:rest} we may assume that $X$ is affine and admits an \'etale morphism $X\to \A^d$ as any nonsingular variety has a basis of such Zariski open subsets.
%
%
We apply \cite[Prop.5.2]{Ach17} to conclude that there is then a finite \'etale map $\pi: X\to \A^d$.  By Lemma~\ref{lem:cpt} we may assume that $\bar X$ is the normalisation of $\P^d$ in $k(X)$.  We apply Lemma~\ref{lem:proj} 1) to $j_Y$ being $\A^d\hookrightarrow \P^d$. 
Assume  there is a natural number $M$ such that for any $\sG\in  \sS(j_Y, r\delta, D')$, there is a finite separable  morphism
$Y_{\sG}\to Y=\A^d$ such that $\sG|_{  Y_{\sG}}$ is tame outside of codimension $2$.  For $\sG=\pi_*\sF, \ \sF \in  \sS(j, r, D)$, as $\pi^*\sG $ surjects to $\sF$, we conclude that  $\sF|_{   Y_{\sG}\times_{\A^d} X }$  is tame outside of codimension $2$  as well. On the other hand, 
 the morphism $Y_{\sG}\times_{\A^d} X \to X$
  is separable and of degree $\le M$, so is a connected component of it. We choose a random one which we define to be $X_{\sF}$.   This finishes the proof.

\end{proof}

\section{Proof of Theorem~\ref{thm:bdd} and Corollary~\ref{cor:bdd}}
The aim of this section is to prove Theorem~\ref{thm:bdd} and Corollary~\ref{cor:bdd}. We start with a simple lemma. 
If $G$ is a finite group we denote by $|G|$ its cardinality. For a natural number $n$, we set $n!=\prod_{i=1}^n i$.  Let $p$ be a prime number. If $A,B$ are subgroups of $G$, where $B$ 
normalizes $A$, we write $A\cdot B$ for the subgroup of $G$ consisting of the image of $A\times B$ under $(a,b)\mapsto ab$.

\begin{lemma}  \label{lem:gr}
 Let
 \ga{}{1\to K \to G \to G/K\to 1\notag}
  be an exact sequence of groups, and $H\subset G$ be a subgroup of $G$ with the following properties:
 \begin{enumerate}
  \item[(i)] The composite map $H \subset G\to  G/K$ is an isomorphism (so $G=K\cdot H$).
  \item[(ii)] $K$ is finite, and $Q$ is the unique $p$-Sylow subgroup of $K$.
  \item[(iii)] $K/Q$ is cyclic. 
 \end{enumerate}
Then there exists a normal subgroup $N\lhd G$ such that 
\begin{itemize}
\item[1)] 
$N\cap Q=\{1\}$;
\item[2)]
$|G/N| \le (|Q|(|Q|!))!$.
\end{itemize}
 \end{lemma}
\begin{proof}
By (ii) $Q$ is a  characteristic subgroup of $K$, thus also a normal subgroup of $G$, as  $K$ is  normal   in $ G$.
Let  
\ga{}{ Z_G(Q)={\rm Ker} \big(G\to {\rm Aut}(Q), \ g\mapsto [x\mapsto gxg^{-1}]\big) \notag}
 be the centraliser of $Q$ in $G$. It is a normal subgroup of $G$ of index $\le |Q|!$.
By  (iii) applying the Schur-Zassenhaus theorem,  $K=Q\cdot K'$ where $K'$ is finite cyclic of order prime to 
$p$ (the decomposition need not be unique). Thus
\ga{}{ K\cap (Q\cdot Z_G(Q))= Q\cdot (K\cap Z_G(Q))= Q\times R, \notag}
where $R$ is cyclic, of order prime to $p$, and now this direct product decomposition is unique (here $R\subset K\cap Z_G(Q)$ is isomorphic to the image of $K\cap Z_G(Q)$ in the cyclic group $K/Q$).
As  $ Z_G(Q)$ is  normal in  $G$, $K\cap (Q\cdot Z_G(Q))=Q\times R$ is normal in $G$, in particular
 $H$ acts by conjugation on it, respecting the direct  product  decomposition.
Hence $R\cdot H \subset G$ is a subgroup, by (i) of index  equal to the index  of $R\subset K$, which is at most $|Q|(|Q|!)$.
 Moreover $Q\cap R\cdot H=\{1\}$. 
\ga{}{
N={\displaystyle\mathop{\bigcap_{g\in G}}}g(R\cdot H)g^{-1} 
\notag}
is a normal subgroup of $G$, of index at most $(|Q|(|Q|!))!$, and has trivial intersection with $Q$.
\end{proof}

\begin{proof}[Proof of Theorem~\ref{thm:bdd}]
By Proposition~\ref{prop:A} we may assume that $j: X=\A^d\hookrightarrow \bar X= \P^d$. We set $Z=\P^d\setminus \A^d$. 
The local ring $\sO_{\bar X,Z} $  is isomorphic to  the local ring  $\sO_{\A^1/k(Z),0}$. The choice of a parameter $t$ on $\A^1$ identifies it with $k(Z)[t]_{(t)}$.  
The inclusion  $\sK\cong k(Z)(t) \subset \hat \sK \cong k(Z)((t))$ of the field of fraction of  $\sO_{\bar X,Z} $ into its completion defines $\hat{X}\to \G_m/k(Z)$ where  $\hat{X}=\Spec \hat{\sK}$.

\medskip

\noindent
Choosing an algebraic closure $\overline{\hat \sK}$ defines a geometric point $x\to \hat X\to \G_m/k(Z)$. 
  One has the diagram of  exact sequences
\ga{}{\xymatrix{
& & 1 \ar[d]  & 1\ar[d] \\
1 \ar[r] &\ar[d]_{=}  P \ar[r] & I \ar[d] \ar[r] & I^{\rm tame} \ar[d] \ar[r] & 1 \\
1\ar[r] &  P\ar[r] & \ar[d] \pi_1(\hat X,x) \ar[r]& \ar[d]  \pi_1^{\rm tame}(\hat X,x) \ar[r] & 1\\
& & \pi_1(k(Z), x)  \ar[d]  \ar[r]^{=} & \pi_1(k(Z), x) \ar[d] \\
& &  1 & 1
}
 \notag}
where the  groups on the right are the tame quotients of $I$ and $\pi_1(\hat X,x)$.

\medskip
\noindent
  By  \cite[Thm.1.4.1]{Kat86} one has a splitting
\ga{1}{ \xymatrix{ \ar[dr]_{\cong} \pi_1( \hat X, x) \ar[r] & \pi_1(\G_m/k(Z),x) \ar[d] \\
  &    \pi_1(\G_m/k(Z),x)[sp]
  }
  }
  where $[sp]$ stands for `special', and is defined by the property that finite quotients have a unique $p$-Sylow. 
  On the other hand, the choice of the rational point  ${\rm Spec}(k(Z))\to \G_m/k(Z)$ defined by $t=1$ 
  yields a splitting $\pi_1(k(Z), x)  \to \pi_1(\G_m/k(Z),x)$, thus a splitting 
  $\pi_1(k(Z), x)  \to \pi_1(\G_m/k(Z),x)[sp]$, thus a splitting 
   \ga{}{ \pi_1(k(Z), x)  \xrightarrow{ [t] } \pi_1(\hat X,x). \notag}

   \noindent
   For $\hat \sF$ a $\bar \Q_\ell$-local system on $\hat X$,  we choose a representative $\rho: \pi_1(\hat X,x)\to GL_r(\bar \Q_\ell)$. The
   residual representation $\bar \rho:  \pi_1(\hat X,x)\to GL_r(\bar \F_\ell)$ is defined up to semi-simplification and isomorphism.  We choose one.
   As $I\cap {\rm Ker} \bar \rho$ is normal in $\pi_1(\hat X,x)$, setting 
   \ga{}{ K= \bar \rho(I), \ \psi: \pi_1(\hat X,x)\to G=\pi_1(\hat X,x)/ I\cap {\rm Ker} \bar \rho, \notag}
   one has a commutative diagram of exact sequences
   \ga{}{\xymatrix{ 1 \ar[r] & \ar[d]_{\bar \rho} I \ar[r]& \ar[r]  \ar@/^2pc/[dd]^>>>>{\bar \rho}  \pi_1(\hat X,x) \ar[d]^\psi & \ar[d]^{=} \pi_1(k(Z),x) \ar[r]&  1\\
   1 \ar[r] & K \ar[r]  & \ar[d] G \ar[r] &  \pi_1(k(Z),x) \ar[r]  & 1\\
   & & GL_k(\bar \F_\ell)
   }\notag
 }  
The bottom exact sequence is split by $\psi\circ [t]: \pi_1(k(Z),x)\to G$. 
 We set 
  \ga{}{  \bar \rho(P)=Q,   \ H= \big( \psi\circ   [t] ( \pi_1(k(Z), x) )\big)\subset G. \notag}
   We  apply Lemma~\ref{lem:gr}. 
 Let $N\subset G$ be the normal subgroup of $G$ defined in {\it loc.cit.}. By 2)  the finite Galois cover 
 $\hat X_N\to \hat X$ defined by the composite $ \pi_1(\hat X,x)\xrightarrow{\bar \rho}  G \to G/N$ has  degree $\le (|Q|(|Q|!)!$  and by 1)  it has  the property that the pullback  of $\hat \sF$ to 
  $\hat X_N$ is tame.  On the other hand, via $\eqref{1}$,
    there is a Galois cover $(\G_m/k(Z))_N\to \G_m/k(Z)$ which restricts to 
$\hat X_N\to \hat X$, of the same degree.

  \medskip
  \noindent
 We now globalize the construction on $\A^d$.  Let  $(x_0=t, x_1,\ldots, x_d) \in  H^0(\P^d, \sO(1))$ be a system of coordinates so $t=0$ defines $\P^d\setminus \A^d$.  The choice of this system of coordinates defines the factorization
 \ga{}{ \hat X \to \G_m/k(Z)\to \A^d . \notag}
 The right morphism induces an isomorphism on the field of fractions $\sK$. 
 We define $\sK_N$ to be the function field of $(\G_m/k(Z))_N$, so that $\sK\subset \sK_N$ is a finite Galois extension. Let  $X_{\sF}$ be the normalisation of $X$ in $\sK_N$. Then $X_{\sF}\to X$ has degree  $\le ( |Q| |Q|!)!$ and 
 for $\hat \sF$ being the restriction of $\sF\in \sS(j, r, D)$, the pullback of $\sF$ to $X_{\sF}$ is tame. 
 It remains to show that $|Q|$ is bounded  if $\hat \sF$ is the restriction  to $\hat X$ of $\sF\in \sS(j,r,D)$.

  \medskip
  \noindent
  Let $C\hookrightarrow X$ be a curve in good position with respect to $Z \times_{\bar X} \bar Y $ in the sense of \cite[Defn.7]{EKS19} where $\bar Y\to \bar X$ is the cover associated to  the choice of a residual representation of $\sF$. 
  Then by  Section 4 ($\star$) of {\it loc. cit.}  the image of the inertia at the points of $\bar C\setminus C$ of $\bar \rho|_C$ is equal to $K$, thus the image of the wild  inertia
  is equal to $Q$.  We now apply Proposition 2 of {\it loc.cit.}. This finishes the proof.

\end{proof}

\begin{proof}[Proof of Corollary~\ref{cor:bdd}] Let $\sF\in \sS(j,r,D)$ and $\rho: \pi_1^{\rm \acute{e}t}(X,x)\to GL_r( \sO)$ representing it, where 
$\sO^{\oplus r}\subset \bar \Q_\ell$ is a lattice stabilized by $\rho$.  Let $\frak{m}\subset \sO$ be the maximal ideal. Then by  \cite[Cor. I.6.3.4]{BOU}, for any smooth irreducible $C\hookrightarrow X$ through $x$, the induced representation $\rho|_C: \pi_1^{\rm \acute{e}t}(C,x)\to \pi_1^{\rm \acute{e}t}(X,x)\xrightarrow{\rho} GL_r( \sO)$ has the same image if and only if it has after post-composing with $\pi: GL_r( \sO) \to GL_r(\sO/\frak{m}^2)$ (see \cite[Lem.B.2, Proof]{EK12}). 
 By \cite[Thm.6.3 (iv)]{Jou83}, fixing an ample $|\sH|$ linear system on $\bar X$, there is a dense open subset of $|\sH|^{d-1}$ such that for any closed point $(f_1,\ldots, f_{d-1})$ in it, the corresponding subscheme $\bar C=(f_1)\cap \ldots \cap (f_{d-1})$ is a smooth complete intersection curve in good position with respect to $\bar X\setminus X$,   its pullback to $W\to X$ is connected, its pullback  $C'$ to  $\bar X_{\sF} $  is 
 integral and in good position with respect to $\bar X_{\sF}\setminus X_{\sF}$.
 Here $W\to X$ is the finite \'etale cover associated to  $\pi\circ \rho$. Thus 
 by 
 \cite[Lem.B.2]{EK12}, $\sF|_{C}$ has the same monodromy group as $\sF$. By \cite[Lem.C.2]{Dri12} applied to $C'$, $\sF|_{C'}$, which is tame, has the same monodromy as $\sF|_{X_{\sF}}$.
 This finishes the proof.

\end{proof}

\section{Rank one}

\begin{definition}
If $\sS$ is a given family of $\bar{\Q_{\ell}}$-local systems on $X$, a smooth connected projective curve $ C\subset  X$ which is the complete intersection of generically smooth very ample divisors in $\bar X$ in good position with respect to $\bar X\setminus X$ is called {\em good for $\sS$} if the restriction map $\sF\to \sF|_C$ induces an isomorphism on the monodromy groups for any $\sF \in \sS$.

\end{definition}
We rephrase Theorem~\ref{thm:rankone}:
\begin{theorem} \label{thm:rk1} Given $j, D$, there is a good curve for 
$ \sS(j,1,D)$.
\end{theorem}
Let $n$ be a positive natural number, and $\sS(j,1,D,n)\subset \sS(j, 1, D)$ be the subset of isomorphism classes of rank one $\sF$ with monodromy group $\Z/p^n$. 

\begin{lemma} \label{lem:n}
 Given $j, D$, 
a good curve for $\cup_{\N\ni n\ge 1} \sS(j,1,D,n)$ is a good curve for $\sS(j,1,D)$.

\end{lemma} 
\begin{proof}
Let $\sF \in \sS(j,1,D)$ with underlying  continuous representation $\rho: \pi^{\rm \acute{e}t}_1(X)\to \sO_E^\times $ where $\sO_E$ is a  finite normal extention of $\Z_\ell$.  Let $\frak{m}\subset \sO_E$ be the maximal ideal. 
Let  $\bar \rho$ be the residual 
 representation $\bar \rho: \pi_1^{\rm \acute{e}t}(X)\to  \sO_E^\times \to (\sO_E/\frak{m})^\times$.
Then  $\bar \rho(\pi_1^{\rm \acute{e}t}(X))=\Z/p^n\Z\times \Z/N\Z$ for $n, N$ positive natural numbers with $(N,p)=1$. 
 So  $\rho(\pi_1^{\rm \acute{e}t}(X))= \Z/p^n\times G \subset \sO_E^\times $ where $G$ is a profinite group of pro-order prime to  $p$. 
By \cite[Prop.C2]{Dri12}, any smooth connected projective curve $ C\subset  X$ which is the complete intersection of very ample divisors in $\bar X$ in good position with respect to $\bar X\setminus X$
is good for all tame $\bar \Q_\ell$-local systems. Thus any such $C$ is good for 
 $\rho^t: \pi^{\rm \acute{e}t}_1(X) \xrightarrow{\rho} \Z/p^n\Z\times G\xrightarrow{{\rm projection}} G\subset  \sO_E^\times$. 
 This proves the lemma. \end{proof}

\begin{lemma} \label{lem:nto1}
Given $j, D$, a good curve for   $\sS(j,1, D,1)$ is a good curve for   $\cup_{\N\ni n\ge 1} \sS(j,1,D,n)$.

\end{lemma}

\begin{proof}
Let $n\ge 2$ and $\rho: \pi^{\rm \acute{e}t}_1(X) \to \Z/p^n\Z$ a surjective representation. Let $C$ be a good curve for $\bar \rho: \pi^{\rm \acute{e}t}_1(X) \to \Z/p^n\Z\to \Z/p\Z$. As a  homomorphism $K\to \Z/p^n\Z$, where $K$ is any group,  is surjective if and only if the composed homomorphism $K\to \Z/p^n\Z\to \Z/p\Z$ is, the lemma follows.

\end{proof}

In order to prove Theorem~\ref{thm:rk1} we may assume  that $X$ is affine. Let us denote by $F: X\to X$ the absolute Frobenius. For an effective Cartier divisor $\Delta$ with support in $\bar X\setminus X$, we define $A(\bar X, \Delta)$ to be the image of $H^0(\bar X, \sO(\Delta))$ in $H^1_{\rm \acute{e}t}(X, \Z/p\Z)=H^0(X, \sO)/(F-1)H^0(X, \sO)$ where  this description of $H^1_{\rm \acute{e}t}(X, \Z/p\Z)$ follows from Artin-Schreier theory. 

\begin{lemma} \label{lem:ESZ} Given a natural number $m\ge 1$, 
there is a normal generically smooth  very ample divisor $\bar H \subset \bar X$ in good position with respect $\bar X\setminus X$ such that 
the restriction homomorphism $A(\bar X,  npD)\to A(\bar H, \bar H\cap npD)$ is  an isomorphism  for all  natural number $1\le n\le m$. There is curve $\bar C\to \bar X$  which is a complete intersection of such $H$ such that $A(\bar X, pnD) \to A(\bar C, \bar C\cdot pnD)$ is  an isomorphism 
for all  natural number $1\le n\le m$.
\end{lemma}
\begin{proof} We choose $\bar H$ a normal generically smooth  very ample divisor $\bar H \subset \bar X$ in good position with respect $\bar X\setminus X$ such that  
\ml{}{ 
H^i(\bar X, \sO_{\bar X}(-H +nD'))=0, \ {\rm for}  \ D'=D \ {\rm or}  \ D'=pD,  \ i=0,1, \\ {\rm and \ for \ all} \ 1\le n\le m. \notag}
The existence of a normal very ample $\bar H$ in good position follows from \cite{Sei50}, and Bertini's theorem.  The cohomology vanishing property holds for $i=0,1$ by the Enriques-Severi-Zariski lemma \cite[Cor.7.8]{Har77}, as there are finitely many such $n$, and the invertible sheaves $\sO_{\bar X}(n D')$ for $D'=D$ or $pD$  are $S_2$.  It then follows that the restriction maps
\ga{}{
 H^0(\bar{X},\sO( nD'))\stackrel{\cong}{\longrightarrow} H^0( \bar H,\sO( nD' \cap \bar H)) \notag\\
 (F-1)H^0(\bar X, \sO(nD)) \xrightarrow{ \cong} (F-1)H^0(\bar H, \sO(nD \cap \bar H))
\notag}
are isomorphisms for $1\le n\le m$, and thus 
\ga{}{ A(\bar{X}, npD)\to A(\bar{H}, npD\cap \bar{H})\notag}
as well. 
This finishes the quest for $\bar H$.
For $\bar C$ we iterate the argument with $\bar X$ replaced by $\bar H$ etc. This finishes the proof. 

\end{proof}

Let $Y$ be a smooth connected variety defined over $k$,  $j_Y:Y\hookrightarrow \bar Y$ be a good semi-compactification, that is, $\bar Y$ is smooth connected and $\bar Y\setminus Y$ is a normal crossings divisor. Given an effective divisor $D$ supported on $\bar Y\setminus Y$, we define $\sS(j_Y,r, D), \ \sS(j_Y, 1, D, n) $ etc. analogously. We denote by 
$ \lfloor \frac{D}{p} \rfloor   $  the floor of the $\Q$-divisor $\frac{D}{p}$, that is the largest Cartier divisor $\Delta$ on $Y$ such that $p\Delta\subset D$. Following \cite[3-2]{KS14} we define (with slightly simplified notation)
the complex $ \Z/p\Z_D$ on $\bar Y_{\rm \acute{e}t}$ by the formula
\ga{}{ \Z/p\Z_D={\rm cone} (\sO( \lfloor \frac{D}{p}  \rfloor) \xrightarrow{F-1} \sO(D))[-1].\notag}
\begin{lemma} \label{lem:prop25}
One has $\sS(j_Y,1,D,1)= H^1_{\rm \acute{e}t}(\bar Y, \Z/p\Z_D).$
\end{lemma}
\begin{proof}
If $\bar Y$ is affine, then 
\ga{}{H^1_{\rm \acute{e}t}(\bar Y, \Z/p\Z_D)    = H^0(\bar Y, \sO(D))/(F-1) H^0(\bar Y, \sO( \lfloor \frac{D}{p} \rfloor),      \notag}
so it follows from \cite[Prop.2.5]{KS14}. In general, let $\sH^i$ be the Zariski sheaf associated to $\bar U\to H^i(\bar Y\cap \bar U , \Z/p\Z_D)$. For $i=0$  it is equal to the constant sheaf $\Z/p\Z$ thus 
\ga{}{ H^1_{\rm \acute{e}t}(\bar Y, \Z/p\Z_D)=H^0_{\rm Zar}(\bar Y, \sH^1).\notag}
This finishes the proof.

\end{proof}
We come back to $j: X\hookrightarrow \bar X$ a projective normal connected compactification of a smooth variety and denote by $\bar Y$ its  smooth locus. 
We set $S=\bar X\setminus \bar Y$ for the singular locus, defining the factorization $j: X\to \bar X\setminus S\xrightarrow{\alpha} \bar X$.  We extend to $\bar X_{\rm \acute{e}t}$ the definition of $\Z/p\Z_D$ on  $\bar Y_{\rm \acute{e}t}$ for $p$-divisible Cartier divisors supported  on $\bar X\setminus X$, as follows.
Let $D$ be an effective Cartier divisor supported on $\bar X\setminus X$. For any natural number $m$, we define the complex  $\Z/p\Z_{mpD}$
on $\bar X_{\rm \acute{e}t}$  by the formula
\ga{}{ \Z/p\Z_{mpD}={\rm cone} (\sO( mD) \xrightarrow{F-1} \sO(mpD))[-1].\notag}
For an effective divisor $D'$ on $\bar Y$ and any $m$,  we have the maps
 \ga{}{ {\rm on}  \ \bar X_{\rm \acute{e}t}:  \Z/p\Z_{mpD}\to \Z/p\Z_{(m+1)pD}, \  {\rm on} \ \bar Y_{\rm \acute{e}t}:  \Z/p\Z_{D}\to \Z/p\Z_{D+D'} \notag}
   {\it raising the level}
  stemming from the natural inclusions $\sL\to \sL(\Delta)$ for a line bundle $\sL$ and an effective Cartier divisor $\Delta$. 
\begin{lemma} \label{lem:homalg} 
For any $m\ge 2$, the level raising maps
\ga{}{ H^1_{\rm{\acute{e}t}}(\bar X, \Z/p\Z_{(m-1)pD})\to H^1_{\rm{\acute{e}t}}(\bar X, \Z/p\Z_{mpD}),  \ 
H^1_{\rm{\acute{e}t}}(\bar Y, \Z/p\Z_D)\to H^1_{\rm{\acute{e}t}}(\bar Y, \Z/p\Z_{D+{D'}})\notag} 
are injective, and so are the restriction maps
\ga{}{ H^1_{\rm{\acute{e}t}}(\bar X, \Z/p\Z_{mpD})\to H^1_{\rm{\acute{e}t}}(\bar Y, \Z/p\Z_{mpD}) \to H^1_{\rm{\acute{e}t}}(X, \Z/p\Z) \notag} for any $m\ge 0$.
\end{lemma}
\begin{proof}
For the level raising map we compute $H^1_{et}(-, \Z/p\Z_{?})$ in the Zariski topology  as $H^1(-, \sO( \lfloor \frac{\Delta}{p} \rfloor)\xrightarrow{F-1} \sO(\Delta))$  for an effective Cartier divisor $\Delta$  and argue that if $(F-1)\psi=\varphi$ for local sections $\varphi\in \sO(\Delta)$ and
$\psi\in \sO( * \Delta)$
then $\psi\in \sO( \lfloor \frac{\Delta}{p} \rfloor)$. Here $\frac{\Delta}{p}$ is Cartier in the $\bar X$-version, on the smooth $\bar Y$ we allow is to be a $\Q$-divisor.  For the restriction map this is the same proof noting that $S\subset \bar X$ has codimension $\ge 2$ thus $\bar Y\cap Y$ hits all the codimension $1$ points of $\bar X\setminus X$.

\end{proof}

\begin{prop} \label{prop:m} There is a natural number  $m\ge 1$ such that in the diagram
\ga{}{ \xymatrix{  &  H^1_{\rm{\acute{e}t}}(\bar Y, \Z/p\Z_D) \ar@{^{(}->}[d] \\
H^1_{\rm{\acute{e}t}}(\bar X, \Z/p\Z_{mpD}) \ar@{^{(}->}[r]  & H^1_{\rm{\acute{e}t}}(\bar Y, \Z/p\Z_{mpD})
}\notag}
the image of $H^1_{\rm{\acute{e}t}}(\bar Y, \Z/p\Z_D)$ by the level raising map falls in the image of 
$H^1_{\rm{\acute{e}t}}(\bar X, \Z/p\Z_{mpD}) $  by the restriction map.
\end{prop}

\begin{proof}
By Lemma~\ref{lem:homalg} we may replace $H^1_{\rm{\acute{e}t}}(\bar Y, \Z/p\Z_D)$ by $H^1_{\rm{\acute{e}t}}(\bar Y, \Z/p\Z_{pD})$ in the statement.  For any $m\ge 1$ we have an exact sequence
\ml{}{ 0\to H^1_{\rm{\acute{e}t}}(\bar X, \Z/p\Z_{pmD})\to H^1_{\rm{\acute{e}t}}(\bar Y, \Z/p\Z_{pmD})\to \\ Q_m= {\rm Ker}(H^0(\bar X, R^1\alpha_*\sO\otimes (mD))\xrightarrow{F-1} 
H^0(\bar X, R^1\alpha_*\sO\otimes (pmD))) \notag}
so it is enough to show that the level raising map $Q_1\to Q_m$ is zero for $m$ large.
Let $W\subset S$ be the non Cohen-Macaulay locus of $\bar{X}$. It  is closed and $S\setminus W\hookrightarrow S$ is dense  (\cite[Cor.6.11.3]{EGAIV(2)}). This defines the factorization
\ga{}{ \alpha: \bar X\setminus S\xrightarrow{\beta} \bar X\setminus W\xrightarrow{\gamma} \bar X\notag}
inducing the exact sequence 
\ga{}{ (\star) \ \ 0\to R^1\gamma_* \sO\to R^1\alpha_*\sO\to \gamma_* R^1\beta_* \sO .\notag }
For any variety $Z$, we denote by $Z^{(i)}$, resp. $Z^{(\ge i)}$   the set of codimension $i$, resp $\ge i$  points. 
By  \cite[ VIII, Cor.2.3]{SGA2}, since
\[\inf_{x\in \bar{X}\setminus W}\{{\rm depth}\,\sO_x+\dim\overline{\{x\}}-\dim W\cap \overline{\{x\}}\}>2,\]
which holds as $\bar{X}\setminus W$ is Cohen-Macaulay, and $W$ has codimension $\ge 3$ in $\bar{X}$,
the Zariski sheaf 
$R^1\gamma_*\sO=  \sH^2_W(\bar X, \sO)$ is coherent. So it is annihilated by a power $\sI_W^n$ of the ideal sheaf of $W$, thus by $\sO_{\bar X} (-mD)$ with $m$ large chosen so $\sI^n_W\supset \sO_{\bar X}(-mD)$.  The level raising map on $Q_a$ for some $a\ge 1$ respects the filtration defined by $(\star)$.  Thus for $a\ge m$ it sends $Q_1$ to 
\ga{}{ Q^{CM}_a={\rm Ker}  (H^0(\bar X, R^1\gamma_*\sO\otimes (aD))\xrightarrow{F-1} 
H^0(\bar X, R^1\gamma_*\sO\otimes (paD))) \subset Q_a \notag }
and we have to show that for $b$ large, the level raising map $Q_a^{CM}\to Q_{a+b}^{CM}$ is zero.
Points 
 $x\in \bar X^{(\ge c)} \cap (S\setminus W)$ have depth $\ge c$, thus  by \cite[III, Lem.~3.1, Prop.~3.3]{SGA2} $H^i_x(\bar X, \sL)=0$ for $i<c$,  and any invertible $\sL$ on $\bar{X}$. 
 Thus the restriction map to the finitely many points in $(\bar X)^{(2)}\cap ( S\setminus W)$ 
\ga{}{ R^1\beta_*\sO\hookrightarrow 
\oplus_{x\in (\bar X)^{(2)}\cap (\bar X\setminus W)} H^2_x( \bar X, \sO)\notag}
is injective. Denoting by $\sO_x$ the local ring of $\bar X$ at $x$, in which $D$ is defined by $g\in \sO_x$, the vertical arrows define isomorphism of complexes 
\ga{}{\xymatrix{\ar[d]_{g^a}  \sO(aD)_x \ar[r]^{F-1} & \sO(paD)_x \ar[d]^{g^{ap}}\\
\sO_x \ar[r]_{F-g^{ap-a}\cdot 1} & \sO_x
} \notag}

We set 
\ga{}{ Q_a(x)={\rm Ker}(H^2_x(\sO)\otimes \sO(aD) \xrightarrow{F-1}
 H^2_x(\sO)\otimes \sO(paD)) \notag\\
Q'_a(x)={\rm Ker}(H^2_x(\sO_x) \xrightarrow{F-g^{ap-a}\cdot 1} H^2_x(\sO_x)). \notag }
For every $b\ge 0$ we have a commutative diagram 
\ga{}{ \xymatrix{ \ar[d]_{g^a} Q_a(x) \ar[r]^{1} &  Q_{a+b}(x) \ar[d]^{g^{a+b}} \\
\sO_x\cdot Q'_a(x) \ar[r]^{g^b} & \sO_x\cdot Q'_{a+b}(x) }\notag}
where the upper horizontal arrow is the level raising map, the vertical maps are injective  and  $\sO_x\cdot Q'_m(x)$ is the $\sO_x$-submodule  of $H^2_x(\sO_x)$ generated by $Q'_m(x)$.
By \cite[Rmk.4.4]{Smi94},  $H^2_x(\sO_x)$ has one maximal Frobenius invariant  proper
$\sO_x$-submodule $K_x$,  the tight closure of $0$ in $H^2_x(\sO_x)$,  which in addition is  finite. On the other hand,   $F(\lambda y)=(\lambda^pg^{ap-a}) y$ for $\lambda\in \sO_x, y\in Q'_a(x)$, so that $\sO_x\cdot y$ has finite $\sO_x$-length, and is Frobenius invariant, for each such $y$; thus 
$\sO_x\cdot y\subset K_x$ for each such $y$. Hence $ \sO_x\cdot Q'_a(x) \subset K_x$ for any $a$.
Hence $\sO_x\cdot Q'_a(x)$ is annihilated by a power of the maximal ideal $\frak{m}^n_{ x} \subset \sO_x$, thus by $g^b$ for $b$ so large that $\frak{m}^n_{ x} \supset \sO_{\bar X}(-bD)$.  Thus the level raising map is zero for $b$ large. This finishes the proof.

\end{proof}

\begin{proof}[Proof of Theorem~\ref{thm:rk1}] By enlarging $D$ we may assume that $X$ is affine. By Lemma~\ref{lem:n} and Lemma~\ref{lem:nto1} it is enough to find a good curve for $\sS(j,1,\Delta,1)$ with $\Delta=mpD$ 
as in Proposition~\ref{prop:m}.
For $\bar C$ as in Lemma~\ref{lem:ESZ}, 
$A(\bar X, \Delta) \to A(\bar C, \bar C\cdot \Delta)$ is injective and by 
Proposition~\ref{prop:m} 
$A(\bar X, \Delta)\supset \sS(j,1, D,1)$. This finishes the proof.

\end{proof}

\begin{remark} \label{rmk:deligne}  We denote by $\sS(j_Y, 1, D\cap \bar Y)\supset \sS(j,1, D)$ the set of rank $1$ $\bar \Q_\ell$ local systems with ramification bounded by $D\cap \bar Y $ on $\bar Y$
where
$\bar Y$ denotes the smooth locus of $\bar X$ as in Proposition~\ref{prop:m}, and $D$ is a Cartier divisor on $\bar X$. 
Then in fact the curve $C$ constructed in Theorem~\ref{thm:rk1} is good for  $\sS(j_Y, 1,  D \cap \bar Y)$ as well. This enables to sharpen Deligne's question as to whether in higher rank there is a curve good for $\sS(j_Y,r,D \cap \bar Y) \supset  \sS(j,r,D)$.

\end{remark}

\section{Remarks}
Let $k, X, \bar X$  be as in  Section~\ref{intro}, such that  $\bar X$ is projective.
For any field extension $K\supset k$, if  
 $C_K$ is a smooth geometrically connected curve over $K$, 
 we denote  by 
$ \bar \eta \mapsto C_K$ a geometric point,  and by $\bar{K}$ an algebraic closure of $K$ in $K(\bar \eta)$.  
We make the following elementary remark.

\begin{remark} \label{rmk:lefschetz}
There is a purely transcendental extension $K/k$  of finite type  and a morphism $C_K\to X$  of a smooth  geometrically connected curve over $K$
such that the composed homomorphism
\ga{}{\pi_1^{\rm \acute{e}t}(C_{\bar{K}}, \bar \eta)\to \pi_1^{\rm \acute{e}t}(C_{K}, \bar \eta)\to
\pi_1^{\rm \acute{e}t}(X,\bar \eta) \notag}  is surjective. 
\end{remark}
In other words, the Lefschetz theorem in the strong form is true if we allow finite type field extensions, even only purely transcendental ones, and then the algebraic closure. So Deligne's problem whether or not the Lefschetz theorem is true after bounding $r$ and $D$ is really a question over  the original algebraically closed field of definition.

\begin{proof}

We may assume $X$ is affine, by passing to a dense  open subset if necessary. Let $f=(f_1,\ldots,f_n):X\to \A^n_k$ be an embedding. In particular, $f$ is unramified. 
As in \cite[Section 6.5]{Jou83},  we consider 
\[Z^{(d-1)}_X\subset X\times_k (\A^{n+1})^{d-1}\]
defined by
\[u_{i0}+\sum_{j=1}^n u_{ij}f_j(x)=0,\;\;1\leq i\leq d-1,\]
where $(u_{ij})_{0\leq j\leq n}$ are the coordinate functions on the $i$-th copy of $\A^{n+1}$. By \cite[Thm.~6.6]{Jou83}, the projection $\pi_X:Z^{(d-1)}_X\to \A^{(n+1)(d-1)}_k$ is dominant. 
 By \cite[6.5.3]{Jou83}, the other projection $Z^{(d-1)}_X\to X$ is a trivial fibre  bundle  with fibre  $(\A^n)^{d-1}$.  In particular $Z^{(d-1)}_X$ is smooth over $k$, thus
 the generic fibre $C_K$ of 
 $\pi_X$ is smooth over $K=k(  \A^{(n+1)(d-1)}_k )$.
 As $\dim f(X)\ge (d-1)+1=d$,      $C_K$ 
 is  geometrically integral by  \cite[Thm.~6.6 (3)]{Jou83}, of dimension 
$\dim Z^{(d-1)}_X-(n+1)(d-1)=1$. 
By construction  $C_K\subset X_K$ is a closed embedding. 

\medskip
\noindent
To prove that $C_K$ has the desired property, we have to show that if $\varphi:Y\to X$ is any connected finite \'etale covering, then $Y\times_XC_{\bar K}$ is connected. Let  $f_Y=f\circ \varphi:Y\to \A^n_k$ be the composition, which is again unramified. Then the construction of \cite[Section 6.5]{Jou83} applied to $f_Y$ gives rise to a similarly defined 
$Z^{(d-1)}_Y\subset Y\times_k (\A^{n+1})^{d-1}$
and a corresponding smooth and geometrically integral curve $C(Y)_K$, which is the geometric generic fibre of $Z^{(d-1)}_Y\to (\A^{n+1})^{d-1}$. By definition
\ga{}{ Z^{(d-1)}_Y=Z^{(d-1)}_X\times_XY, \ Y_K \supset C(Y)_K =C_K\times_XY=C_K\times_{X_K}Y_K.\notag}
Hence as claimed, the inverse image $C(Y)_{\bar K}$ of $C_{\bar K}$ under $\varphi\times_k \bar K: Y_{\bar K}\to X_{\bar K}$ is integral, in particular connected. This finishes the proof. 

\end{proof}

\bibliographystyle{amsalpha}
 
\end{document}